\documentclass{amsart}
\usepackage{amsfonts}

\newtheorem{theorem}{Theorem}
\theoremstyle{plain}

\newtheorem{corollary}{Corollary}

\newtheorem{definition}{Definition}

\newtheorem{lemma}{Lemma}

\newtheorem{proposition}{Proposition}
\newtheorem{remark}{Remark}

\numberwithin{equation}{section}
\input{tcilatex}

\begin{document}
\title[Inequalities for Quantum $f$-Divergence]{Inequalities for Quantum $f$%
-Divergence of Trace Class Operators in Hilbert Spaces }
\author{S.S. Dragomir$^{1,2}$}
\address{$^{1}$Mathematics, School of Engineering \& Science\\
Victoria University, PO Box 14428\\
Melbourne City, MC 8001, Australia.}
\email{sever.dragomir@vu.edu.au}
\urladdr{http://rgmia.org/dragomir}
\address{$^{2}$School of Computational \& Applied Mathematics, University of
the Witwatersrand, Private Bag 3, Johannesburg 2050, South Africa}

\begin{abstract}
Some inequalities for quantum $f$-divergence of trace class operators in
Hilbert spaces are obtained. It is shown that for normalised convex
functions it is nonnegative. Some upper bounds for quantum $f$-divergence in
terms of variational and $\chi ^{2}$-distance are provided. Applications for
some classes of divergence measures such as Umegaki and Tsallis relative
entropies are also given.
\end{abstract}

\maketitle

\section{Introduction}

Let $\left( X,\mathcal{A}\right) $ be a measurable space satisfying $%
\left\vert \mathcal{A}\right\vert >2$ and $\mu $ be a $\sigma $-finite
measure on $\left( X,\mathcal{A}\right) .$ Let $\mathcal{P}$ be the set of
all probability measures on $\left( X,\mathcal{A}\right) $ which are
absolutely continuous with respect to $\mu .$ For $P,$ $Q\in \mathcal{P}$,
let $p=\frac{dP}{d\mu }$ and $q=\frac{dQ}{d\mu }$ denote the \textit{%
Radon-Nikodym }derivatives of $P$ and $Q$ with respect to $\mu .$

Two probability measures $P,$ $Q\in \mathcal{P}$ are said to be \textit{%
orthogonal} and we denote this by $Q\perp P$ if%
\begin{equation*}
P\left( \left\{ q=0\right\} \right) =Q\left( \left\{ p=0\right\} \right) =1.
\end{equation*}

Let $f:[0,\infty )\rightarrow (-\infty ,\infty ]$ be a convex function that
is continuous at $0,$ i.e., $f\left( 0\right) =\lim_{u\downarrow 0}f\left(
u\right) .$

In 1963, I. Csisz\'{a}r \cite{IC1} introduced the concept of $f$-divergence
as follows.

\begin{definition}
\label{d1.1}Let $P,$ $Q\in \mathcal{P}$. Then%
\begin{equation}
I_{f}\left( Q,P\right) =\int_{X}p\left( x\right) f\left[ \frac{q\left(
x\right) }{p\left( x\right) }\right] d\mu \left( x\right) ,  \label{1.1}
\end{equation}%
is called the $f$-divergence of the probability distributions $Q$ and $P.$
\end{definition}

\begin{remark}
\label{r.1}Observe that, the integrand in the formula (\ref{1.1}) is
undefined when $p\left( x\right) =0.$ The way to overcome this problem is to
postulate for $f$ as above that%
\begin{equation}
0f\left[ \frac{q\left( x\right) }{0}\right] =q\left( x\right)
\lim_{u\downarrow 0}\left[ uf\left( \frac{1}{u}\right) \right] ,\text{ }x\in
X.  \label{1.1.a}
\end{equation}
\end{remark}

We now give some examples of $f$-divergences that are well-known and often
used in the literature (see also \cite{CDO}).

\subsection{The Class of $\protect\chi ^{\protect\alpha }$-Divergences}

The $f$-divergences of this class, which is generated by the function $\chi
^{\alpha },$ $\alpha \in \lbrack 1,\infty ),$ defined by%
\begin{equation*}
\chi ^{\alpha }\left( u\right) =\left\vert u-1\right\vert ^{\alpha },\ \ \
u\in \lbrack 0,\infty )
\end{equation*}%
have the form%
\begin{equation}
I_{f}\left( Q,P\right) =\int_{X}p\left\vert \frac{q}{p}-1\right\vert
^{\alpha }d\mu =\int_{X}p^{1-\alpha }\left\vert q-p\right\vert ^{\alpha
}d\mu .  \label{1.2}
\end{equation}%
From this class only the parameter $\alpha =1$ provides a distance in the
topological sense, namely the \textit{total variation distance }$V\left(
Q,P\right) =\int_{X}\left\vert q-p\right\vert d\mu .$ The most prominent
special case of this class is, however, \textit{Karl Pearson's }$\chi ^{2}$-%
\textit{divergence }%
\begin{equation*}
\chi ^{2}\left( Q,P\right) =\int_{X}\frac{q^{2}}{p}d\mu -1
\end{equation*}%
that is obtained for $\alpha =2.$

\subsection{Dichotomy Class}

From this class, generated by the function $f_{\alpha }:[0,\infty
)\rightarrow \mathbb{R}$%
\begin{equation*}
f_{\alpha }\left( u\right) =\left\{ 
\begin{array}{ll}
u-1-\ln u & \text{for\ \ }\alpha =0; \\ 
&  \\ 
\frac{1}{\alpha \left( 1-\alpha \right) }\left[ \alpha u+1-\alpha -u^{\alpha
}\right] & \text{for\ \ }\alpha \in \mathbb{R}\backslash \left\{ 0,1\right\}
; \\ 
&  \\ 
1-u+u\ln u & \text{for\ \ }\alpha =1;%
\end{array}%
\right.
\end{equation*}%
only the parameter $\alpha =\frac{1}{2}$ $\left( f_{\frac{1}{2}}\left(
u\right) =2\left( \sqrt{u}-1\right) ^{2}\right) $ provides a distance,
namely, the \textit{Hellinger distance}%
\begin{equation*}
H\left( Q,P\right) =\left[ \int_{X}\left( \sqrt{q}-\sqrt{p}\right) ^{2}d\mu %
\right] ^{\frac{1}{2}}.
\end{equation*}

Another important divergence is the \textit{Kullback-Leibler divergence}
obtained for $\alpha =1,$%
\begin{equation*}
KL\left( Q,P\right) =\int_{X}q\ln \left( \frac{q}{p}\right) d\mu .
\end{equation*}

\subsection{Matsushita's Divergences}

The elements of this class, which is generated by the function $\varphi
_{\alpha },$ $\alpha \in (0,1]$ given by%
\begin{equation*}
\varphi _{\alpha }\left( u\right) :=\left\vert 1-u^{\alpha }\right\vert ^{%
\frac{1}{\alpha }},\ \ \ u\in \lbrack 0,\infty ),
\end{equation*}%
are prototypes of metric divergences, providing the distances $\left[
I_{\varphi _{\alpha }}\left( Q,P\right) \right] ^{\alpha }.$

\subsection{Puri-Vincze Divergences}

This class is generated by the functions $\Phi _{\alpha },$ $\alpha \in
\lbrack 1,\infty )$ given by%
\begin{equation*}
\Phi _{\alpha }\left( u\right) :=\frac{\left\vert 1-u\right\vert ^{\alpha }}{%
\left( u+1\right) ^{\alpha -1}},\ \ \ u\in \lbrack 0,\infty ).
\end{equation*}%
It has been shown in \cite{KOV} that this class provides the distances $%
\left[ I_{\Phi _{\alpha }}\left( Q,P\right) \right] ^{\frac{1}{\alpha }}.$

\subsection{Divergences of Arimoto-type}

This class is generated by the functions%
\begin{equation*}
\Psi _{\alpha }\left( u\right) :=\left\{ 
\begin{array}{ll}
\frac{\alpha }{\alpha -1}\left[ \left( 1+u^{\alpha }\right) ^{\frac{1}{%
\alpha }}-2^{\frac{1}{\alpha }-1}\left( 1+u\right) \right] & \text{for\ \ }%
\alpha \in \left( 0,\infty \right) \backslash \left\{ 1\right\} ; \\ 
&  \\ 
\left( 1+u\right) \ln 2+u\ln u-\left( 1+u\right) \ln \left( 1+u\right) & 
\text{for\ \ }\alpha =1; \\ 
&  \\ 
\frac{1}{2}\left\vert 1-u\right\vert & \text{for\ \ }\alpha =\infty .%
\end{array}%
\right.
\end{equation*}%
It has been shown in \cite{OV} that this class provides the distances $\left[
I_{\Psi _{\alpha }}\left( Q,P\right) \right] ^{\min \left( \alpha ,\frac{1}{%
\alpha }\right) }$ for $\alpha \in \left( 0,\infty \right) $ and $\frac{1}{2}%
V\left( Q,P\right) $ for $\alpha =\infty .$

For $f$ continuous convex on $[0,\infty )$ we obtain the $\ast $-\textit{%
conjugate }function of $f$ by%
\begin{equation*}
f^{\ast }\left( u\right) =uf\left( \frac{1}{u}\right) ,\ \ \ u\in \left(
0,\infty \right)
\end{equation*}%
and%
\begin{equation*}
f^{\ast }\left( 0\right) =\lim_{u\downarrow 0}f^{\ast }\left( u\right) .
\end{equation*}%
It is also known that if $f$ is continuous convex on $[0,\infty )$ then so
is $f^{\ast }.$

The following two theorems contain the most basic properties of $f$%
-divergences. For their proofs we refer the reader to Chapter 1 of \cite{LV}
(see also \cite{CDO}).

\begin{theorem}[Uniqueness and Symmetry Theorem]
\label{t3.1}Let $f,f_{1}$ be continuous convex on $[0,\infty ).$ We have%
\begin{equation*}
I_{f_{1}}\left( Q,P\right) =I_{f}\left( Q,P\right) ,
\end{equation*}%
for all $P,Q\in \mathcal{P}$ if and only if there exists a constant $c\in 
\mathbb{R}$ such that%
\begin{equation*}
f_{1}\left( u\right) =f\left( u\right) +c\left( u-1\right) ,
\end{equation*}%
for any $u\in \lbrack 0,\infty ).$
\end{theorem}

\begin{theorem}[Range of Values Theorem]
\label{t3.2}Let $f:[0,\infty )\rightarrow \mathbb{R}$ be a continuous convex
function on $[0,\infty ).$

For any $P,Q\in \mathcal{P}$, we have the double inequality%
\begin{equation}
f\left( 1\right) \leq I_{f}\left( Q,P\right) \leq f\left( 0\right) +f^{\ast
}\left( 0\right) .  \label{3.1}
\end{equation}

\begin{enumerate}
\item[(i)] If $P=Q,$ then the equality holds in the first part of (\ref{3.1}%
).
\end{enumerate}

If $f$ is strictly convex at $1,$ then the equality holds in the first part
of (\ref{3.1}) if and only if $P=Q;$

\begin{enumerate}
\item[(ii)] If $Q\perp P,$ then the equality holds in the second part of (%
\ref{3.1}).
\end{enumerate}

If $f\left( 0\right) +f^{\ast }\left( 0\right) <\infty ,$ then equality
holds in the second part of (\ref{3.1}) if and only if $Q\perp P.$
\end{theorem}

The following result is a refinement of the second inequality in Theorem \ref%
{t3.2} (see \cite[Theorem 3]{CDO}).

\begin{theorem}
\label{t3.3}Let $f$ be a continuous convex function on $[0,\infty )$ with $%
f\left( 1\right) =0$ ($f$ is normalised) and $f\left( 0\right) +f^{\ast
}\left( 0\right) <\infty .$ Then%
\begin{equation}
0\leq I_{f}\left( Q,P\right) \leq \frac{1}{2}\left[ f\left( 0\right)
+f^{\ast }\left( 0\right) \right] V\left( Q,P\right)  \label{3.2}
\end{equation}%
for any $Q,P\in \mathcal{P}$.
\end{theorem}

For other inequalities for $f$-divergence see \cite{CD}, \cite{SSD11}-\cite%
{SSD1}.

Motivated by the above results, in this paper we obtain some new
inequalities for quantum $f$-divergence of trace class operators in Hilbert
spaces. It is shown that for normalised convex functions it is nonnegative.
Some upper bounds for quantum $f$-divergence in terms of variational and $%
\chi ^{2}$-distance are provided. Applications for some classes of
divergence measures such as Umegaki and Tsallis relative entropies are also
given.

In what follows we recall some facts we need concerning the trace of
operators and quantum $f$-divergence for trace class operators in infinite
dimensional complex Hilbert spaces.

\section{Some Preliminary Facts}

\subsection{Some Facts on Trace of Operators}

Let $\left( H,\left\langle \cdot ,\cdot \right\rangle \right) $ be a complex
Hilbert space and $\left\{ e_{i}\right\} _{i\in I}$ an \textit{orthonormal
basis} of $H.$ We say that $A\in \mathcal{B}\left( H\right) $ is a \textit{%
Hilbert-Schmidt operator} if%
\begin{equation}
\sum_{i\in I}\left\Vert Ae_{i}\right\Vert ^{2}<\infty .  \label{e.1.1.a}
\end{equation}%
It is well know that, if $\left\{ e_{i}\right\} _{i\in I}$ and $\left\{
f_{j}\right\} _{j\in J}$ are orthonormal bases for $H$ and $A\in \mathcal{B}%
\left( H\right) $ then%
\begin{equation}
\sum_{i\in I}\left\Vert Ae_{i}\right\Vert ^{2}=\sum_{j\in I}\left\Vert
Af_{j}\right\Vert ^{2}=\sum_{j\in I}\left\Vert A^{\ast }f_{j}\right\Vert ^{2}
\label{e.1.2.a}
\end{equation}%
showing that the definition (\ref{e.1.1.a}) is independent of the
orthonormal basis and $A$ is a Hilbert-Schmidt operator iff $A^{\ast }$ is a
Hilbert-Schmidt operator.

Let $\mathcal{B}_{2}\left( H\right) $ the set of Hilbert-Schmidt operators
in $\mathcal{B}\left( H\right) .$ For $A\in \mathcal{B}_{2}\left( H\right) $
we define%
\begin{equation}
\left\Vert A\right\Vert _{2}:=\left( \sum_{i\in I}\left\Vert
Ae_{i}\right\Vert ^{2}\right) ^{1/2}  \label{e.1.3.1}
\end{equation}%
for $\left\{ e_{i}\right\} _{i\in I}$ an orthonormal basis of $H.$ This
definition does not depend on the choice of the orthonormal basis.

Using the triangle inequality in $l^{2}\left( I\right) ,$ one checks that $%
\mathcal{B}_{2}\left( H\right) $ is a \textit{vector space} and that $%
\left\Vert \cdot \right\Vert _{2}$ is a norm on $\mathcal{B}_{2}\left(
H\right) ,$ which is usually called in the literature as the \textit{%
Hilbert-Schmidt norm}.

Denote \textit{the modulus} of an operator $A\in \mathcal{B}\left( H\right) $
by $\left\vert A\right\vert :=\left( A^{\ast }A\right) ^{1/2}.$

Because $\left\Vert \left\vert A\right\vert x\right\Vert =\left\Vert
Ax\right\Vert $ for all $x\in H,$ $A$ is Hilbert-Schmidt iff $\left\vert
A\right\vert $ is Hilbert-Schmidt and $\left\Vert A\right\Vert
_{2}=\left\Vert \left\vert A\right\vert \right\Vert _{2}.$ From (\ref%
{e.1.2.a}) we have that if $A\in \mathcal{B}_{2}\left( H\right) ,$ then $%
A^{\ast }\in \mathcal{B}_{2}\left( H\right) $ and $\left\Vert A\right\Vert
_{2}=\left\Vert A^{\ast }\right\Vert _{2}.$

The following theorem collects some of the most important properties of
Hilbert-Schmidt operators:

\begin{theorem}
\label{t.1.1}We have:

(i) $\left( \mathcal{B}_{2}\left( H\right) ,\left\Vert \cdot \right\Vert
_{2}\right) $ is a Hilbert space with inner product 
\begin{equation}
\left\langle A,B\right\rangle _{2}:=\sum_{i\in I}\left\langle
Ae_{i},Be_{i}\right\rangle =\sum_{i\in I}\left\langle B^{\ast
}Ae_{i},e_{i}\right\rangle  \label{e.1.4.1}
\end{equation}%
and the definition does not depend on the choice of the orthonormal basis $%
\left\{ e_{i}\right\} _{i\in I}$;

(ii) We have the inequalities 
\begin{equation}
\left\Vert A\right\Vert \leq \left\Vert A\right\Vert _{2}  \label{e.1.4.a}
\end{equation}%
for any $A\in \mathcal{B}_{2}\left( H\right) $ and 
\begin{equation}
\left\Vert AT\right\Vert _{2},\text{ }\left\Vert TA\right\Vert _{2}\leq
\left\Vert T\right\Vert \left\Vert A\right\Vert _{2}  \label{e.1.4.b}
\end{equation}%
for any $A\in \mathcal{B}_{2}\left( H\right) $ and $T\in \mathcal{B}\left(
H\right) ;$

(iii) $\mathcal{B}_{2}\left( H\right) $ is an operator ideal in $\mathcal{B}%
\left( H\right) ,$ i.e. 
\begin{equation*}
\mathcal{B}\left( H\right) \mathcal{B}_{2}\left( H\right) \mathcal{B}\left(
H\right) \subseteq \mathcal{B}_{2}\left( H\right) ;
\end{equation*}

(iv) $\mathcal{B}_{fin}\left( H\right) ,$ the space of operators of finite
rank, is a dense subspace of $\mathcal{B}_{2}\left( H\right) ;$

(v) $\mathcal{B}_{2}\left( H\right) \subseteq \mathcal{K}\left( H\right) ,$
where $\mathcal{K}\left( H\right) $ denotes the algebra of compact operators
on $H.$
\end{theorem}

If $\left\{ e_{i}\right\} _{i\in I}$ an orthonormal basis of $H,$ we say
that $A\in \mathcal{B}\left( H\right) $ is \textit{trace class} if 
\begin{equation}
\left\Vert A\right\Vert _{1}:=\sum_{i\in I}\left\langle \left\vert
A\right\vert e_{i},e_{i}\right\rangle <\infty .  \label{e.1.5.1}
\end{equation}%
The definition of $\left\Vert A\right\Vert _{1}$ does not depend on the
choice of the orthonormal basis $\left\{ e_{i}\right\} _{i\in I}.$ We denote
by $\mathcal{B}_{1}\left( H\right) $ the set of trace class operators in $%
\mathcal{B}\left( H\right) .$

The following proposition holds:

\begin{proposition}
\label{p.1.1}If $A\in \mathcal{B}\left( H\right) ,$ then the following are
equivalent:

(i) $A\in \mathcal{B}_{1}\left( H\right) ;$

(ii) $\left\vert A\right\vert ^{1/2}\in \mathcal{B}_{2}\left( H\right) ;$

(ii) $A$ (or $\left\vert A\right\vert )$ is the product of two elements of $%
\mathcal{B}_{2}\left( H\right) .$
\end{proposition}

The following properties are also well known:

\begin{theorem}
\label{t.1.2}With the above notations:

(i) We have 
\begin{equation}
\left\Vert A\right\Vert _{1}=\left\Vert A^{\ast }\right\Vert _{1}\text{ and }%
\left\Vert A\right\Vert _{2}\leq \left\Vert A\right\Vert _{1}
\label{e.1.6.1}
\end{equation}%
for any $A\in \mathcal{B}_{1}\left( H\right) ;$

(ii) $\mathcal{B}_{1}\left( H\right) $ is an operator ideal in $\mathcal{B}%
\left( H\right) ,$ i.e. 
\begin{equation*}
\mathcal{B}\left( H\right) \mathcal{B}_{1}\left( H\right) \mathcal{B}\left(
H\right) \subseteq \mathcal{B}_{1}\left( H\right) ;
\end{equation*}

(iii) We have%
\begin{equation*}
\mathcal{B}_{2}\left( H\right) \mathcal{B}_{2}\left( H\right) =\mathcal{B}%
_{1}\left( H\right) ;
\end{equation*}

(iv) We have%
\begin{equation*}
\left\Vert A\right\Vert _{1}=\sup \left\{ \left\langle A,B\right\rangle _{2}%
\text{ }|\text{ }B\in \mathcal{B}_{2}\left( H\right) ,\text{ }\left\Vert
B\right\Vert \leq 1\right\} ;
\end{equation*}

(v) $\left( \mathcal{B}_{1}\left( H\right) ,\left\Vert \cdot \right\Vert
_{1}\right) $ is a Banach space.

(iv) We have the following isometric isomorphisms%
\begin{equation*}
\mathcal{B}_{1}\left( H\right) \cong K\left( H\right) ^{\ast }\text{ and }%
\mathcal{B}_{1}\left( H\right) ^{\ast }\cong \mathcal{B}\left( H\right) ,
\end{equation*}%
where $K\left( H\right) ^{\ast }$ is the dual space of $K\left( H\right) $
and $\mathcal{B}_{1}\left( H\right) ^{\ast }$ is the dual space of $\mathcal{%
B}_{1}\left( H\right) .$
\end{theorem}

We define the \textit{trace} of a trace class operator $A\in \mathcal{B}%
_{1}\left( H\right) $ to be%
\begin{equation}
\limfunc{tr}\left( A\right) :=\sum_{i\in I}\left\langle
Ae_{i},e_{i}\right\rangle ,  \label{e.1.7.1}
\end{equation}%
where $\left\{ e_{i}\right\} _{i\in I}$ an orthonormal basis of $H.$ Note
that this coincides with the usual definition of the trace if $H$ is
finite-dimensional. We observe that the series (\ref{e.1.7.1}) converges
absolutely and it is independent from the choice of basis.

The following result collects some properties of the trace:

\begin{theorem}
\label{t.3.1.a}We have:

(i) If $A\in \mathcal{B}_{1}\left( H\right) $ then $A^{\ast }\in \mathcal{B}%
_{1}\left( H\right) $ and 
\begin{equation}
\limfunc{tr}\left( A^{\ast }\right) =\overline{\limfunc{tr}\left( A\right) };
\label{e.1.8.1}
\end{equation}

(ii) If $A\in \mathcal{B}_{1}\left( H\right) $ and $T\in \mathcal{B}\left(
H\right) ,$ then $AT,$ $TA\in \mathcal{B}_{1}\left( H\right) $ and%
\begin{equation}
\limfunc{tr}\left( AT\right) =\limfunc{tr}\left( TA\right) \text{ and }%
\left\vert \limfunc{tr}\left( AT\right) \right\vert \leq \left\Vert
A\right\Vert _{1}\left\Vert T\right\Vert ;  \label{e.1.9.1}
\end{equation}

(iii) $\limfunc{tr}\left( \cdot \right) $ is a bounded linear functional on $%
\mathcal{B}_{1}\left( H\right) $ with $\left\Vert \limfunc{tr}\right\Vert
=1; $

(iv) If $A,$ $B\in \mathcal{B}_{2}\left( H\right) $ then $AB,$ $BA\in 
\mathcal{B}_{1}\left( H\right) $ and $\limfunc{tr}\left( AB\right) =\limfunc{%
tr}\left( BA\right) ;$

(v) $\mathcal{B}_{fin}\left( H\right) $ is a dense subspace of $\mathcal{B}%
_{1}\left( H\right) .$
\end{theorem}

Utilising the trace notation we obviously have that 
\begin{equation*}
\left\langle A,B\right\rangle _{2}=\limfunc{tr}\left( B^{\ast }A\right) =%
\limfunc{tr}\left( AB^{\ast }\right) \text{ and }\left\Vert A\right\Vert
_{2}^{2}=\limfunc{tr}\left( A^{\ast }A\right) =\limfunc{tr}\left( \left\vert
A\right\vert ^{2}\right)
\end{equation*}%
for any $A,$ $B\in \mathcal{B}_{2}\left( H\right) .$

The following H\"{o}lder's type inequality has been obtained by Ruskai in 
\cite{R} 
\begin{equation}
\left\vert \limfunc{tr}\left( AB\right) \right\vert \leq \limfunc{tr}\left(
\left\vert AB\right\vert \right) \leq \left[ \limfunc{tr}\left( \left\vert
A\right\vert ^{1/\alpha }\right) \right] ^{\alpha }\left[ \limfunc{tr}\left(
\left\vert B\right\vert ^{1/\left( 1-\alpha \right) }\right) \right]
^{1-\alpha }  \label{1.9.2}
\end{equation}%
where $\alpha \in \left( 0,1\right) $ and $A,$ $B\in \mathcal{B}\left(
H\right) $ with $\left\vert A\right\vert ^{1/\alpha },$ $\left\vert
B\right\vert ^{1/\left( 1-\alpha \right) }\in \mathcal{B}_{1}\left( H\right)
.$

In particular, for $\alpha =\frac{1}{2}$ we get the Schwarz inequality%
\begin{equation}
\left\vert \limfunc{tr}\left( AB\right) \right\vert \leq \limfunc{tr}\left(
\left\vert AB\right\vert \right) \leq \left[ \limfunc{tr}\left( \left\vert
A\right\vert ^{2}\right) \right] ^{1/2}\left[ \limfunc{tr}\left( \left\vert
B\right\vert ^{2}\right) \right] ^{1/2}  \label{1.9.3}
\end{equation}%
with $A,$ $B\in \mathcal{B}_{2}\left( H\right) .$

If $A\geq 0$ and $P\in \mathcal{B}_{1}\left( H\right) $ with $P\geq 0,$ then 
\begin{equation}
0\leq \limfunc{tr}\left( PA\right) \leq \left\Vert A\right\Vert \limfunc{tr}%
\left( P\right) .  \label{1.9.4}
\end{equation}

Indeed, since $A\geq 0,$ then $\left\langle Ax,x\right\rangle \geq 0$ for
any $x\in H.$ If $\left\{ e_{i}\right\} _{i\in I}$ an orthonormal basis of $%
H $, then%
\begin{equation*}
0\leq \left\langle AP^{1/2}e_{i},P^{1/2}e_{i}\right\rangle \leq \left\Vert
A\right\Vert \left\Vert P^{1/2}e_{i}\right\Vert ^{2}=\left\Vert A\right\Vert
\left\langle Pe_{i},e_{i}\right\rangle
\end{equation*}%
for any $i\in I.$ Summing over $i\in I$ we get%
\begin{equation*}
0\leq \sum_{i\in I}\left\langle AP^{1/2}e_{i},P^{1/2}e_{i}\right\rangle \leq
\left\Vert A\right\Vert \sum_{i\in I}\left\langle Pe_{i},e_{i}\right\rangle
=\left\Vert A\right\Vert \limfunc{tr}\left( P\right)
\end{equation*}%
and since%
\begin{equation*}
\sum_{i\in I}\left\langle AP^{1/2}e_{i},P^{1/2}e_{i}\right\rangle
=\sum_{i\in I}\left\langle P^{1/2}AP^{1/2}e_{i},e_{i}\right\rangle =\limfunc{%
tr}\left( P^{1/2}AP^{1/2}\right) =\limfunc{tr}\left( PA\right)
\end{equation*}%
we obtain the desired result (\ref{1.9.4}).

This obviously imply the fact that, if $A$ and $B$ are selfadjoint operators
with $A\leq B$ and $P\in \mathcal{B}_{1}\left( H\right) $ with $P\geq 0,$
then%
\begin{equation}
\limfunc{tr}\left( PA\right) \leq \limfunc{tr}\left( PB\right) .
\label{1.9.5}
\end{equation}

Now, if $A$ is a selfadjoint operator, then we know that%
\begin{equation*}
\left\vert \left\langle Ax,x\right\rangle \right\vert \leq \left\langle
\left\vert A\right\vert x,x\right\rangle \text{ for any }x\in H.
\end{equation*}%
This inequality follows by Jensen's inequality for the convex function $%
f\left( t\right) =\left\vert t\right\vert $ defined on a closed interval
containing the spectrum of $A.$

If $\left\{ e_{i}\right\} _{i\in I}$ is an orthonormal basis of $H$, then%
\begin{eqnarray}
\left\vert \limfunc{tr}\left( PA\right) \right\vert &=&\left\vert \sum_{i\in
I}\left\langle AP^{1/2}e_{i},P^{1/2}e_{i}\right\rangle \right\vert \leq
\sum_{i\in I}\left\vert \left\langle AP^{1/2}e_{i},P^{1/2}e_{i}\right\rangle
\right\vert  \label{1.9.6} \\
&\leq &\sum_{i\in I}\left\langle \left\vert A\right\vert
P^{1/2}e_{i},P^{1/2}e_{i}\right\rangle =\limfunc{tr}\left( P\left\vert
A\right\vert \right) ,  \notag
\end{eqnarray}%
for any $A$ a selfadjoint operator and $P\in \mathcal{B}_{1}^{+}\left(
H\right) :=\left\{ P\in \mathcal{B}_{1}\left( H\right) \text{ with }P\geq
0\right\} .$

For the theory of trace functionals and their applications the reader is
referred to \cite{Si}.

For some classical trace inequalities see \cite{Ch}, \cite{C}, \cite{N} and 
\cite{Y1}, which are continuations of the work of Bellman \cite{B}. For
related works the reader can refer to \cite{A}, \cite{BJL}, \cite{Ch}, \cite%
{FL}, \cite{Le}, \cite{Li}, \cite{Ma}, \cite{SA0} and \cite{UT}.

\subsection{Quantum $f$-Divergence for Trace Class Operators}

On complex Hilbert space $\left( \mathcal{B}_{2}\left( H\right)
,\left\langle \cdot ,\cdot \right\rangle _{2}\right) ,$ where the
Hilbert-Schmidt inner product is defined by%
\begin{equation*}
\left\langle U,V\right\rangle _{2}:=\limfunc{tr}\left( V^{\ast }U\right) ,%
\text{ }U,\text{ }V\in \mathcal{B}_{2}\left( H\right) ,
\end{equation*}%
for $A,$ $B\in \mathcal{B}^{+}\left( H\right) $ consider the operators $%
\mathfrak{L}_{A}:\mathcal{B}_{2}\left( H\right) \rightarrow \mathcal{B}%
_{2}\left( H\right) $ and $\mathfrak{R}_{B}:\mathcal{B}_{2}\left( H\right)
\rightarrow \mathcal{B}_{2}\left( H\right) $ defined by 
\begin{equation*}
\mathfrak{L}_{A}T:=AT\text{ and }\mathfrak{R}_{B}T:=TB.
\end{equation*}%
We observe that they are well defined and since%
\begin{equation*}
\left\langle \mathfrak{L}_{A}T,T\right\rangle _{2}=\left\langle
AT,T\right\rangle _{2}=\limfunc{tr}\left( T^{\ast }AT\right) =\limfunc{tr}%
\left( \left\vert T^{\ast }\right\vert ^{2}A\right) \geq 0
\end{equation*}%
and 
\begin{equation*}
\left\langle \mathfrak{R}_{B}T,T\right\rangle _{2}=\left\langle
TB,T\right\rangle _{2}=\limfunc{tr}\left( T^{\ast }TB\right) =\limfunc{tr}%
\left( \left\vert T\right\vert ^{2}B\right) \geq 0
\end{equation*}%
for any $T\in \mathcal{B}_{2}\left( H\right) ,$ they are also positive in
the operator order of $\mathcal{B}\left( \mathcal{B}_{2}\left( H\right)
\right) ,$ the Banach algebra of all bounded operators on $\mathcal{B}%
_{2}\left( H\right) $ with the norm $\left\Vert \cdot \right\Vert _{2}$
where $\left\Vert T\right\Vert _{2}=\limfunc{tr}\left( \left\vert
T\right\vert ^{2}\right) ,$ $T\in \mathcal{B}_{2}\left( H\right) .$

Since $\limfunc{tr}\left( \left\vert X^{\ast }\right\vert ^{2}\right) =%
\limfunc{tr}\left( \left\vert X\right\vert ^{2}\right) $ for any $X\in 
\mathcal{B}_{2}\left( H\right) ,$ then also 
\begin{align*}
\limfunc{tr}\left( T^{\ast }AT\right) & =\limfunc{tr}\left( T^{\ast
}A^{1/2}A^{1/2}T\right) =\limfunc{tr}\left( \left( A^{1/2}T\right) ^{\ast
}A^{1/2}T\right) \\
& =\limfunc{tr}\left( \left\vert A^{1/2}T\right\vert ^{2}\right) =\limfunc{tr%
}\left( \left\vert \left( A^{1/2}T\right) ^{\ast }\right\vert ^{2}\right) =%
\limfunc{tr}\left( \left\vert T^{\ast }A^{1/2}\right\vert ^{2}\right)
\end{align*}%
for $A\geq 0$ and $T\in \mathcal{B}_{2}\left( H\right) .$

We observe that $\mathfrak{L}_{A}$ and $\mathfrak{R}_{B}$ are commutative,
therefore the product $\mathfrak{L}_{A}\mathfrak{R}_{B}$ is a selfadjoint
positive operator in $\mathcal{B}\left( \mathcal{B}_{2}\left( H\right)
\right) $ for any positive operators $A,B\in \mathcal{B}\left( H\right) .$

For $A,B\in \mathcal{B}^{+}\left( H\right) $ with $B$ invertible, we define
the \textit{Araki transform} $\mathfrak{A}_{A,B}:\mathcal{B}_{2}\left(
H\right) \rightarrow \mathcal{B}_{2}\left( H\right) $ by $\mathfrak{A}%
_{A,B}:=\mathfrak{L}_{A}\mathfrak{R}_{B^{-1}}.$ We observe that for $T\in 
\mathcal{B}_{2}\left( H\right) $ we have $\mathfrak{A}_{A,B}T=ATB^{-1}$ and 
\begin{equation*}
\left\langle \mathfrak{A}_{A,B}T,T\right\rangle _{2}=\left\langle
ATB^{-1},T\right\rangle _{2}=\limfunc{tr}\left( T^{\ast }ATB^{-1}\right) .
\end{equation*}%
Observe also, by the properties of trace, that 
\begin{align*}
\limfunc{tr}\left( T^{\ast }ATB^{-1}\right) & =\limfunc{tr}\left(
B^{-1/2}T^{\ast }A^{1/2}A^{1/2}TB^{-1/2}\right) \\
& =\limfunc{tr}\left( \left( A^{1/2}TB^{-1/2}\right) ^{\ast }\left(
A^{1/2}TB^{-1/2}\right) \right) =\limfunc{tr}\left( \left\vert
A^{1/2}TB^{-1/2}\right\vert ^{2}\right)
\end{align*}%
giving that%
\begin{equation}
\left\langle \mathfrak{A}_{A,B}T,T\right\rangle _{2}=\limfunc{tr}\left(
\left\vert A^{1/2}TB^{-1/2}\right\vert ^{2}\right) \geq 0  \label{e.2.1}
\end{equation}%
for any $T\in \mathcal{B}_{2}\left( H\right) .$

We observe that, by the definition of operator order and by (\ref{e.2.1}) we
have $r1_{\mathcal{B}_{2}\left( H\right) }\leq \mathfrak{A}_{A,B}\leq R1_{%
\mathcal{B}_{2}\left( H\right) }$ for some $R\geq r\geq 0$ if and only if%
\begin{equation}
r\limfunc{tr}\left( \left\vert T\right\vert ^{2}\right) \leq \limfunc{tr}%
\left( \left\vert A^{1/2}TB^{-1/2}\right\vert ^{2}\right) \leq R\limfunc{tr}%
\left( \left\vert T\right\vert ^{2}\right)  \label{e.2.2}
\end{equation}%
for any $T\in \mathcal{B}_{2}\left( H\right) .$

We also notice that a sufficient condition for (\ref{e.2.2}) to hold is that
the following inequality in the operator order of $\mathcal{B}\left(
H\right) $ is satisfied 
\begin{equation}
r\left\vert T\right\vert ^{2}\leq \left\vert A^{1/2}TB^{-1/2}\right\vert
^{2}\leq R\left\vert T\right\vert ^{2}  \label{e.2.3}
\end{equation}%
for any $T\in \mathcal{B}_{2}\left( H\right) .$

Let $U$ be a selfadjoint linear operator on a complex Hilbert space $\left(
K;\left\langle \cdot ,\cdot \right\rangle \right) .$ The \textit{Gelfand map 
}establishes a $\ast $-isometrically isomorphism $\Phi $ between the set $%
C\left( \limfunc{Sp}\left( U\right) \right) $ of all \textit{continuous
functions} defined on the \textit{spectrum} of $U,$ denoted $\limfunc{Sp}%
\left( U\right) ,$ and the $C^{\ast }$-algebra $C^{\ast }\left( U\right) $
generated by $U$ and the identity operator $1_{K}$ on $K$ as follows:

For any $f,g\in C\left( \limfunc{Sp}\left( U\right) \right) $ and any $%
\alpha ,\beta \in \mathbb{C}$ we have

(i) \ \ \ $\Phi \left( \alpha f+\beta g\right) =\alpha \Phi \left( f\right)
+\beta \Phi \left( g\right) ;$

(ii) \ \ $\Phi \left( fg\right) =\Phi \left( f\right) \Phi \left( g\right) $
and $\Phi \left( \bar{f}\right) =\Phi \left( f\right) ^{\ast };$

(iii) \ $\left\Vert \Phi \left( f\right) \right\Vert =\left\Vert
f\right\Vert :=\sup_{t\in \limfunc{Sp}\left( U\right) }\left\vert f\left(
t\right) \right\vert ;$

(iv) \ \ $\Phi \left( f_{0}\right) =1_{K}$ and $\Phi \left( f_{1}\right) =U,$
where $f_{0}\left( t\right) =1$ and $f_{1}\left( t\right) =t,$ for $t\in 
\limfunc{Sp}\left( U\right) .$

With this notation we define 
\begin{equation*}
f\left( U\right) :=\Phi \left( f\right) \text{\quad for all }f\in C\left( 
\limfunc{Sp}\left( U\right) \right)
\end{equation*}%
and we call it the \textit{continuous functional calculus} for a selfadjoint
operator $U.$

If $U$ is a selfadjoint operator and $f$ is a real valued continuous
function on $\limfunc{Sp}\left( U\right) $, then $f\left( t\right) \geq 0$
for any $t\in \limfunc{Sp}\left( U\right) $ implies that $f\left( U\right)
\geq 0,$ i.e. $f\left( U\right) $ is a positive operator on $K.$ Moreover,
if both $f$ and $g$ are real valued functions on $\limfunc{Sp}\left(
U\right) $ then the following important property holds: 
\begin{equation}
f\left( t\right) \geq g\left( t\right) \text{\quad for any\quad }t\in 
\limfunc{Sp}\left( U\right) \text{\quad implies that\quad }f\left( U\right)
\geq g\left( U\right)  \tag{P}  \label{P}
\end{equation}%
in the operator order of $B\left( K\right) .$

Let $f:[0,\infty )\rightarrow \mathbb{R}$ be a continuous function.
Utilising the continuous functional calculus for the Araki selfadjoint
operator $\mathfrak{A}_{Q,P}\in \mathcal{B}\left( \mathcal{B}_{2}\left(
H\right) \right) $ we can define the \textit{quantum }$f$\textit{-divergence}
for $Q,P\in S_{1}\left( H\right) :=\left\{ P\in \mathcal{B}_{1}\left(
H\right) ,\text{ }P\geq 0\text{ with }\limfunc{tr}\left( P\right) =1\text{ }%
\right\} $ and $P$ invertible, by%
\begin{equation*}
S_{f}\left( Q,P\right) :=\left\langle f\left( \mathfrak{A}_{Q,P}\right)
P^{1/2},P^{1/2}\right\rangle _{2}=\limfunc{tr}\left( P^{1/2}f\left( 
\mathfrak{A}_{Q,P}\right) P^{1/2}\right) .
\end{equation*}

If we consider the continuous convex function $f:[0,\infty )\rightarrow 
\mathbb{R}$, with $f\left( 0\right) :=0$ and $f\left( t\right) =t\ln t$ for $%
t>0$ then for $Q,P\in S_{1}\left( H\right) $ and $Q,P$ invertible we have%
\begin{equation*}
S_{f}\left( Q,P\right) =\limfunc{tr}\left[ Q\left( \ln Q-\ln P\right) \right]
=:U\left( Q,P\right) ,
\end{equation*}%
which is the \textit{Umegaki relative entropy}.

If we take the continuous convex function $f:[0,\infty )\rightarrow \mathbb{R%
}$, $f\left( t\right) =\left\vert t-1\right\vert $ for $t\geq 0$ then for $%
Q,P\in S_{1}\left( H\right) $ with $P$ invertible we have%
\begin{equation*}
S_{f}\left( Q,P\right) =\limfunc{tr}\left( \left\vert Q-P\right\vert \right)
=:V\left( Q,P\right) ,
\end{equation*}%
where $V\left( Q,P\right) $ is the \textit{variational distance}.

If we take $f:[0,\infty )\rightarrow \mathbb{R}$, $f\left( t\right) =t^{2}-1$
for $t\geq 0$ then for $Q,P\in S_{1}\left( H\right) $ with $P$ invertible we
have%
\begin{equation*}
S_{f}\left( Q,P\right) =\limfunc{tr}\left( Q^{2}P^{-1}\right) -1=:\chi
^{2}\left( Q,P\right) ,
\end{equation*}%
which is called the $\chi ^{2}$\textit{-distance}

Let $q\in \left( 0,1\right) $ and define the convex function $%
f_{q}:[0,\infty )\rightarrow \mathbb{R}$ by $f_{q}\left( t\right) =\frac{%
1-t^{q}}{1-q}.$ Then 
\begin{equation*}
S_{f_{q}}\left( Q,P\right) =\frac{1-\limfunc{tr}\left( Q^{q}P^{1-q}\right) }{%
1-q},
\end{equation*}%
which is \textit{Tsallis relative entropy}.

If we consider the convex function $f:[0,\infty )\rightarrow \mathbb{R}$ by $%
f\left( t\right) =\frac{1}{2}\left( \sqrt{t}-1\right) ^{2},$ then 
\begin{equation*}
S_{f}\left( Q,P\right) =1-\limfunc{tr}\left( Q^{1/2}P^{1/2}\right)
=:h^{2}\left( Q,P\right) ,
\end{equation*}%
which is known as \textit{Hellinger discrimination}.

If we take $f:\left( 0,\infty \right) \rightarrow \mathbb{R}$, $f\left(
t\right) =-\ln t$ then for $Q,P\in S_{1}\left( H\right) $ and $Q,P$
invertible we have%
\begin{equation*}
S_{f}\left( Q,P\right) =\limfunc{tr}\left[ P\left( \ln P-\ln Q\right) \right]
=U\left( P,Q\right) .
\end{equation*}%
The reader can obtain other particular quantum $f$-divergence measures by
utilizing the normalized convex functions from Introduction, namely the
convex functions defining the dichotomy class, Matsushita's divergences,
Puri-Vincze divergences or divergences of Arimoto-type. We omit the details.

In the important case of finite dimensional space $H$ and the generalized
inverse $P^{-1},$ numerous properties of the quantum $f$-divergence, mostly
in the case when $f$ is \textit{operator convex,} have been obtained in the
recent papers \cite{HP}, \cite{HMP}, \cite{P1}, \cite{P2}\ and the
references therein.

In what follows we obtain several inequalities for the larger class of
convex functions on an interval.

\section{Inequalities for $f$ Convex and Normalized}

Suppose that $I$ is an interval of real numbers with interior $\mathring{I}$
and $f:I\rightarrow \mathbb{R}$ is a convex function on $I$. Then $f$ is
continuous on $\mathring{I}$ and has finite left and right derivatives at
each point of $\mathring{I}$. Moreover, if $x,y\in \mathring{I}$ and $x<y,$
then $f_{-}^{\prime }\left( x\right) \leq f_{+}^{\prime }\left( x\right)
\leq f_{-}^{\prime }\left( y\right) \leq f_{+}^{\prime }\left( y\right) ,$
which shows that both $f_{-}^{\prime }$ and $f_{+}^{\prime }$ are
nondecreasing function on $\mathring{I}$. It is also known that a convex
function must be differentiable except for at most countably many points.

For a convex function $f:I\rightarrow \mathbb{R}$, the subdifferential of $f$
denoted by $\partial f$ is the set of all functions $\varphi :I\rightarrow %
\left[ -\infty ,\infty \right] $ such that $\varphi \left( \mathring{I}%
\right) \subset \mathbb{R}$ and 
\begin{equation}
f\left( x\right) \geq f\left( a\right) +\left( x-a\right) \varphi \left(
a\right) \text{ for any }x,a\in I.  \tag{G}  \label{G}
\end{equation}

It is also well known that if $f$ is convex on $I,$ then $\partial f$ is
nonempty, $f_{-}^{\prime }$, $f_{+}^{\prime }\in \partial f$ and if $\varphi
\in \partial f$, then 
\begin{equation*}
f_{-}^{\prime }\left( x\right) \leq \varphi \left( x\right) \leq
f_{+}^{\prime }\left( x\right) \text{ for any }x\in \text{$\mathring{I}$.}
\end{equation*}%
In particular, $\varphi $ is a nondecreasing function.

If $f$ is differentiable and convex on $\mathring{I}$, then $\partial
f=\left\{ f^{\prime }\right\} .$

We are able now to state and prove the first result concerning the quantum $f
$-divergence for the general case of convex functions. 

\begin{theorem}
\label{t.2.1}Let $f:[0,\infty )\rightarrow \mathbb{R}$ be a continuous
convex function that is normalized, i.e. $f\left( 1\right) =0.$ Then for any 
$Q,P\in S_{1}\left( H\right) ,$ with $P$ invertible, we have%
\begin{equation}
0\leq S_{f}\left( Q,P\right) .  \label{e.2.4}
\end{equation}%
Moreover, if $f$ is continuously differentiable, then also%
\begin{equation}
S_{f}\left( Q,P\right) \leq S_{\ell f^{\prime }}\left( Q,P\right)
-S_{f^{\prime }}\left( Q,P\right) ,  \label{e.2.5}
\end{equation}%
where the function $\ell $ is defined as $\ell \left( t\right) =t,$ $t\in 
\mathbb{R}$.
\end{theorem}

\begin{proof}
Since $f$ is convex and normalized, then by the gradient inequality (\ref{G}%
) we have%
\begin{equation*}
f\left( t\right) \geq \left( t-1\right) f_{+}^{\prime }\left( 1\right)
\end{equation*}%
for $t>0.$

Applying the property (\ref{P}) for the operator $\mathfrak{A}_{Q,P},$ then
we have for any $T\in \mathcal{B}_{2}\left( H\right) $%
\begin{eqnarray*}
\left\langle f\left( \mathfrak{A}_{Q,P}\right) T,T\right\rangle _{2} &\geq
&f_{+}^{\prime }\left( 1\right) \left\langle \left( \mathfrak{A}_{Q,P}-1_{%
\mathcal{B}_{2}\left( H\right) }\right) T,T\right\rangle _{2} \\
&=&f_{+}^{\prime }\left( 1\right) \left[ \left\langle \mathfrak{A}%
_{Q,P}T,T\right\rangle _{2}-\left\Vert T\right\Vert _{2}\right] ,
\end{eqnarray*}%
which, in terms of trace, can be written as%
\begin{equation}
\limfunc{tr}\left( T^{\ast }f\left( \mathfrak{A}_{Q,P}\right) T\right) \geq
f_{+}^{\prime }\left( 1\right) \left[ \limfunc{tr}\left( \left\vert
Q^{1/2}TP^{-1/2}\right\vert ^{2}\right) -\limfunc{tr}\left( \left\vert
T\right\vert ^{2}\right) \right]  \label{e.2.6}
\end{equation}%
for any $T\in \mathcal{B}_{2}\left( H\right) .$

The inequality (\ref{e.2.6}) is of interest in itself.

Now, if we take in (\ref{e.2.6}) $T=P^{1/2}$ where $P\in S_{1}\left(
H\right) ,$ with $P$ invertible, then we get%
\begin{equation*}
S_{f}\left( Q,P\right) \geq f_{+}^{\prime }\left( 1\right) \left[ \limfunc{tr%
}\left( Q\right) -\limfunc{tr}\left( P\right) \right] =0
\end{equation*}%
and the inequality (\ref{e.2.4}) is proved.

Further, if $f$ is continuously differentiable, then by the gradient
inequality we also have%
\begin{equation*}
\left( t-1\right) f^{\prime }\left( t\right) \geq f\left( t\right)
\end{equation*}%
for $t>0.$

Applying the property (\ref{P}) for the operator $\mathfrak{A}_{Q,P},$ then
we have for any $T\in \mathcal{B}_{2}\left( H\right) $ 
\begin{equation*}
\left\langle \left( \mathfrak{A}_{Q,P}-1_{\mathcal{B}_{2}\left( H\right)
}\right) f^{\prime }\left( \mathfrak{A}_{Q,P}\right) T,T\right\rangle
_{2}\geq \left\langle f\left( \mathfrak{A}_{Q,P}\right) T,T\right\rangle
_{2},
\end{equation*}%
namely 
\begin{equation*}
\left\langle \mathfrak{A}_{Q,P}f^{\prime }\left( \mathfrak{A}_{Q,P}\right)
T,T\right\rangle _{2}-\left\langle f^{\prime }\left( \mathfrak{A}%
_{Q,P}\right) T,T\right\rangle _{2}\geq \left\langle f\left( \mathfrak{A}%
_{Q,P}\right) T,T\right\rangle _{2},
\end{equation*}%
for any $T\in \mathcal{B}_{2}\left( H\right) ,$ or in terms of trace%
\begin{equation}
\limfunc{tr}\left( T^{\ast }\mathfrak{A}_{Q,P}f^{\prime }\left( \mathfrak{A}%
_{Q,P}\right) T\right) -\limfunc{tr}\left( T^{\ast }f^{\prime }\left( 
\mathfrak{A}_{Q,P}\right) T\right) \geq \limfunc{tr}\left( T^{\ast }f\left( 
\mathfrak{A}_{Q,P}\right) T\right) ,  \label{e.2.7}
\end{equation}%
for any $T\in \mathcal{B}_{2}\left( H\right) .$

This inequality is also of interest in itself.

If in (\ref{e.2.7}) we take $T=P^{1/2},$ where $P\in S_{1}\left( H\right) ,$
with $P$ invertible, then we get the desired result (\ref{e.2.5}).
\end{proof}

\begin{remark}
\label{r.2.1}If we take in (\ref{e.2.5}) $f:\left( 0,\infty \right)
\rightarrow \mathbb{R}$, $f\left( t\right) =-\ln t$ then for $Q,P\in
S_{1}\left( H\right) $ and $Q,P$ invertible we have%
\begin{equation}
0\leq U\left( P,Q\right) \leq \chi ^{2}\left( P,Q\right) .  \label{e.2.7.1}
\end{equation}
\end{remark}

We need the following lemma that is of interest in itself.

\begin{lemma}
\label{l.2.1}Let $S$ be a selfadjoint operator on the Hilbert space $\left(
K,\left\langle \cdot ,\cdot \right\rangle \right) $ and with spectrum $%
\limfunc{Sp}\left( S\right) \subseteq \left[ \gamma ,\Gamma \right] $ for
some real numbers $\gamma ,\Gamma .$ If $g:\left[ \gamma ,\Gamma \right]
\rightarrow \mathbb{C}$ is a continuous function such that 
\begin{equation}
\left\vert g\left( t\right) -\lambda \right\vert \leq \rho \text{ for any }%
t\in \left[ \gamma ,\Gamma \right]  \label{e.2.8}
\end{equation}%
for some complex number $\lambda \in \mathbb{C}$ and positive number $\rho ,$
then%
\begin{align}
\left\vert \left\langle Sg\left( S\right) x,x\right\rangle -\left\langle
Sx,x\right\rangle \left\langle g\left( S\right) x,x\right\rangle \right\vert
& \leq \rho \left\langle \left\vert S-\left\langle Sx,x\right\rangle
1_{H}\right\vert x,x\right\rangle  \label{e.2.9} \\
& \leq \rho \left[ \left\langle S^{2}x,x\right\rangle -\left\langle
Sx,x\right\rangle ^{2}\right] ^{1/2}  \notag
\end{align}%
for any $x\in K,$ $\left\Vert x\right\Vert =1.$
\end{lemma}

\begin{proof}
We observe that 
\begin{equation}
\left\langle Sg\left( S\right) x,x\right\rangle -\left\langle
Sx,x\right\rangle \left\langle g\left( S\right) x,x\right\rangle
=\left\langle \left( S-\left\langle Sx,x\right\rangle 1_{H}\right) \left(
g\left( S\right) -\lambda 1_{H}\right) x,x\right\rangle  \label{e.2.10}
\end{equation}%
for any $x\in K,$ $\left\Vert x\right\Vert =1.$

For any selfadjoint operator $B$ we have the modulus inequality%
\begin{equation}
\left\vert \left\langle Bx,x\right\rangle \right\vert \leq \left\langle
\left\vert B\right\vert x,x\right\rangle \text{ for any }x\in K,\left\Vert
x\right\Vert =1.  \label{e.2.11}
\end{equation}%
Also, utilizing the continuous functional calculus we have for each fixed $%
x\in K,\left\Vert x\right\Vert =1$ 
\begin{align*}
\left\vert \left( S-\left\langle Sx,x\right\rangle 1_{H}\right) \left(
g\left( S\right) -\lambda 1_{H}\right) \right\vert & =\left\vert
S-\left\langle Sx,x\right\rangle 1_{H}\right\vert \left\vert g\left(
S\right) -\lambda 1_{H}\right\vert \\
& \leq \rho \left\vert S-\left\langle Sx,x\right\rangle 1_{H}\right\vert ,
\end{align*}%
which implies that%
\begin{equation}
\left\langle \left\vert \left( S-\left\langle Sx,x\right\rangle 1_{H}\right)
\left( g\left( S\right) -\lambda 1_{H}\right) \right\vert x,x\right\rangle
\leq \rho \left\langle \left\vert S-\left\langle Sx,x\right\rangle
1_{H}\right\vert x,x\right\rangle  \label{e.2.12}
\end{equation}%
for any $x\in K,\left\Vert x\right\Vert =1.$

Therefore, by taking the modulus in (\ref{e.2.10}) and utilizing (\ref%
{e.2.11}) and (\ref{e.2.12}) we get%
\begin{align}
& \left\vert \left\langle Sg\left( S\right) x,x\right\rangle -\left\langle
Sx,x\right\rangle \left\langle g\left( S\right) x,x\right\rangle \right\vert
\label{e.2.13} \\
& =\left\vert \left\langle \left( S-\left\langle Sx,x\right\rangle
1_{H}\right) \left( g\left( S\right) -\lambda 1_{H}\right) x,x\right\rangle
\right\vert  \notag \\
& \leq \left\langle \left\vert \left( S-\left\langle Sx,x\right\rangle
1_{H}\right) \left( g\left( S\right) -\lambda 1_{H}\right) \right\vert
x,x\right\rangle  \notag \\
& \leq \rho \left\langle \left\vert S-\left\langle Sx,x\right\rangle
1_{H}\right\vert x,x\right\rangle  \notag
\end{align}%
for any $x\in K,\left\Vert x\right\Vert =1,$ which proves the first
inequality in (\ref{e.2.9}).

Using Schwarz inequality we also have%
\begin{align*}
\left\langle \left\vert S-\left\langle Sx,x\right\rangle 1_{H}\right\vert
x,x\right\rangle & \leq \left\langle \left( S-\left\langle Sx,x\right\rangle
1_{H}\right) ^{2}x,x\right\rangle ^{1/2} \\
& =\left[ \left\langle S^{2}x,x\right\rangle -\left\langle Sx,x\right\rangle
^{2}\right] ^{1/2}
\end{align*}%
for any $x\in K,\left\Vert x\right\Vert =1,$ and the lemma is proved.
\end{proof}

\begin{corollary}
\label{c.2.1}With the assumption of Lemma \ref{l.2.1}, we have%
\begin{eqnarray}
0 &\leq &\left\langle S^{2}x,x\right\rangle -\left\langle Sx,x\right\rangle
^{2}\leq \frac{1}{2}\left( \Gamma -\gamma \right) \left\langle \left\vert
S-\left\langle Sx,x\right\rangle 1_{H}\right\vert x,x\right\rangle
\label{e.2.14} \\
&\leq &\frac{1}{2}\left( \Gamma -\gamma \right) \left[ \left\langle
S^{2}x,x\right\rangle -\left\langle Sx,x\right\rangle ^{2}\right] ^{1/2}\leq 
\frac{1}{4}\left( \Gamma -\gamma \right) ^{2},  \notag
\end{eqnarray}%
for any $x\in K,$ $\left\Vert x\right\Vert =1.$
\end{corollary}

\begin{proof}
If we take in Lemma \ref{l.2.1} $g\left( t\right) =t,$ $\lambda =\frac{1}{2}%
\left( \Gamma +\gamma \right) $ and $\rho =\frac{1}{2}\left( \Gamma -\gamma
\right) ,$ then we get%
\begin{align}
0& \leq \left\langle S^{2}x,x\right\rangle -\left\langle Sx,x\right\rangle
^{2}\leq \frac{1}{2}\left( \Gamma -\gamma \right) \left\langle \left\vert
S-\left\langle Sx,x\right\rangle 1_{H}\right\vert x,x\right\rangle
\label{e.2.15} \\
& \leq \frac{1}{2}\left( \Gamma -\gamma \right) \left[ \left\langle
S^{2}x,x\right\rangle -\left\langle Sx,x\right\rangle ^{2}\right] ^{1/2} 
\notag
\end{align}%
for any $x\in K,$ $\left\Vert x\right\Vert =1.$

From the first and last terms in (\ref{e.2.15}) we have%
\begin{equation*}
\left[ \left\langle S^{2}x,x\right\rangle -\left\langle Sx,x\right\rangle
^{2}\right] ^{1/2}\leq \frac{1}{2}\left( \Gamma -\gamma \right) ,
\end{equation*}%
which proves the rest of (\ref{e.2.14}).
\end{proof}

We can prove the following result that provides simpler upper bounds for the
quantum $f$-divergence when the operators $P$ and $Q$ satisfy the condition (%
\ref{e.2.2}).

\begin{theorem}
\label{t.2.2}Let $f:[0,\infty )\rightarrow \mathbb{R}$ be a continuous
convex function that is normalized. If $Q,P\in S_{1}\left( H\right) ,$ with $%
P$ invertible, and there exists $R\geq 1\geq r\geq 0$ such that%
\begin{equation}
r\limfunc{tr}\left( \left\vert T\right\vert ^{2}\right) \leq \limfunc{tr}%
\left( \left\vert Q^{1/2}TP^{-1/2}\right\vert ^{2}\right) \leq R\limfunc{tr}%
\left( \left\vert T\right\vert ^{2}\right)  \label{e.2.16}
\end{equation}%
for any $T\in \mathcal{B}_{2}\left( H\right) ,$ then 
\begin{align}
0& \leq S_{f}\left( Q,P\right) \leq \frac{1}{2}\left[ f_{-}^{\prime }\left(
R\right) -f_{+}^{\prime }\left( r\right) \right] V\left( Q,P\right)
\label{e.2.17} \\
& \leq \frac{1}{2}\left[ f_{-}^{\prime }\left( R\right) -f_{+}^{\prime
}\left( r\right) \right] \chi \left( Q,P\right)  \notag \\
& \leq \frac{1}{4}\left( R-r\right) \left[ f_{-}^{\prime }\left( R\right)
-f_{+}^{\prime }\left( r\right) \right] .  \notag
\end{align}
\end{theorem}

\begin{proof}
Without loosing the generality, we prove the inequality in the case that $f$
is continuously differentiable on $\left( 0,\infty \right) .$

Since $f^{\prime }$ is monotonic nondecreasing on $\left[ r,R\right] $ we
have that 
\begin{equation*}
f^{\prime }\left( r\right) \leq f^{\prime }\left( t\right) \leq f^{\prime
}\left( R\right) \text{ for any }t\in \left[ r,R\right] ,
\end{equation*}%
which implies that%
\begin{equation*}
\left\vert f^{\prime }\left( t\right) -\frac{f^{\prime }\left( R\right)
+f^{\prime }\left( r\right) }{2}\right\vert \leq \frac{1}{2}\left[ f^{\prime
}\left( R\right) -f^{\prime }\left( r\right) \right]
\end{equation*}%
for any $t\in \left[ r,R\right] .$

Applying Lemma \ref{l.2.1} and Corollary \ref{c.2.1} in the Hilbert space $%
\left( \mathcal{B}_{2}\left( H\right) ,\left\langle \cdot ,\cdot
\right\rangle _{2}\right) $ and for the selfadjoint operator $\mathfrak{A}%
_{Q,P}$ we have%
\begin{align*}
& \left\vert \left\langle \mathfrak{A}_{Q,P}f^{\prime }\left( \mathfrak{A}%
_{Q,P}\right) T,T\right\rangle _{2}-\left\langle \mathfrak{A}%
_{Q,P}T,T\right\rangle _{2}\left\langle f^{\prime }\left( \mathfrak{A}%
_{Q,P}\right) T,T\right\rangle _{2}\right\vert \\
& \leq \frac{1}{2}\left[ f^{\prime }\left( R\right) -f^{\prime }\left(
r\right) \right] \left\langle \left\vert \mathfrak{A}_{Q,P}-\left\langle 
\mathfrak{A}_{Q,P}T,T\right\rangle _{2}1_{\mathcal{B}_{2}\left( H\right)
}\right\vert T,T\right\rangle _{2} \\
& \leq \frac{1}{2}\left[ f^{\prime }\left( R\right) -f^{\prime }\left(
r\right) \right] \left[ \left\langle \mathfrak{A}_{Q,P}^{2}T,T\right\rangle
_{2}-\left\langle \mathfrak{A}_{Q,P}T,T\right\rangle _{2}^{2}\right] ^{1/2}
\\
& \leq \frac{1}{4}\left( R-r\right) \left[ f_{-}^{\prime }\left( R\right)
-f_{+}^{\prime }\left( r\right) \right]
\end{align*}

for any $T\in \mathcal{B}_{2}\left( H\right) ,$ $\left\Vert T\right\Vert
_{2}=1,$ which is an inequality of interest in itself as well.

If in this inequality we take $T=P^{1/2},$ $P\in S_{1}\left( H\right) ,$
with $P$ invertible, then we get%
\begin{align*}
& \left\vert \left\langle \mathfrak{A}_{Q,P}f^{\prime }\left( \mathfrak{A}%
_{Q,P}\right) P^{1/2},P^{1/2}\right\rangle _{2}-\left\langle f^{\prime
}\left( \mathfrak{A}_{Q,P}\right) P^{1/2},P^{1/2}\right\rangle
_{2}\right\vert \\
& \leq \frac{1}{2}\left[ f^{\prime }\left( R\right) -f^{\prime }\left(
r\right) \right] \left\langle \left\vert \mathfrak{A}_{Q,P}-\left\langle 
\mathfrak{A}_{Q,P}P^{1/2},P^{1/2}\right\rangle _{2}1_{\mathcal{B}_{2}\left(
H\right) }\right\vert P^{1/2},P^{1/2}\right\rangle _{2} \\
& \leq \frac{1}{2}\left[ f^{\prime }\left( R\right) -f^{\prime }\left(
r\right) \right] \left[ \left\langle \mathfrak{A}_{Q,P}^{2}P^{1/2},P^{1/2}%
\right\rangle _{2}-\left\langle \mathfrak{A}_{Q,P}P^{1/2},P^{1/2}\right%
\rangle _{2}^{2}\right] ^{1/2} \\
& \leq \frac{1}{4}\left( R-r\right) \left[ f_{-}^{\prime }\left( R\right)
-f_{+}^{\prime }\left( r\right) \right] ,
\end{align*}%
which can be written as 
\begin{align*}
\left\vert S_{\ell f^{\prime }}\left( Q,P\right) -S_{f^{\prime }}\left(
Q,P\right) \right\vert & \leq \frac{1}{2}\left[ f_{-}^{\prime }\left(
R\right) -f_{+}^{\prime }\left( r\right) \right] V\left( Q,P\right) \\
& \leq \frac{1}{2}\left[ f_{-}^{\prime }\left( R\right) -f_{+}^{\prime
}\left( r\right) \right] \chi \left( Q,P\right) \\
& \leq \frac{1}{4}\left( R-r\right) \left[ f_{-}^{\prime }\left( R\right)
-f_{+}^{\prime }\left( r\right) \right] .
\end{align*}%
Making use of Theorem \ref{t.2.1} we deduce the desired result (\ref{e.2.17}%
).
\end{proof}

\begin{remark}
\label{r.2.2} If we take in (\ref{e.2.17}) $f\left( t\right) =t^{2}-1,$ then
we get%
\begin{align}
0& \leq \chi ^{2}\left( Q,P\right) \leq \frac{1}{2}\left( R-r\right) V\left(
Q,P\right) \leq \frac{1}{2}\left( R-r\right) \chi \left( Q,P\right)
\label{e.2.17.a} \\
& \leq \frac{1}{4}\left( R-r\right) ^{2}  \notag
\end{align}%
for $Q,P\in S_{1}\left( H\right) ,$ with $P$ invertible and satisfying the
condition (\ref{e.2.16}).

If we take in (\ref{e.2.17}) $f\left( t\right) =t\ln t$, then we get the
inequality%
\begin{align}
0& \leq U\left( Q,P\right) \leq \frac{1}{2}\ln \left( \frac{R}{r}\right)
V\left( Q,P\right) \leq \frac{1}{2}\ln \left( \frac{R}{r}\right) \chi \left(
Q,P\right)  \label{e.2.18} \\
& \leq \frac{1}{4}\left( R-r\right) \ln \left( \frac{R}{r}\right)  \notag
\end{align}%
provided that $Q,P\in S_{1}\left( H\right) ,$ with $P,$ $Q$ invertible and
satisfying the condition (\ref{e.2.16}).

With the same conditions and if we take $f\left( t\right) =-\ln t,$ then%
\begin{equation}
0\leq U\left( P,Q\right) \leq \frac{R-r}{2rR}V\left( Q,P\right) \leq \frac{%
R-r}{2rR}\chi \left( Q,P\right) \leq \frac{\left( R-r\right) ^{2}}{4rR}.
\label{e.2.19}
\end{equation}

If we take in (\ref{e.2.17}) $f\left( t\right) =f_{q}\left( t\right) =\frac{%
1-t^{q}}{1-q},$ then we get%
\begin{align}
0& \leq S_{f_{q}}\left( Q,P\right) \leq \frac{q}{2\left( 1-q\right) }\left( 
\frac{R^{1-q}-r^{1-q}}{R^{1-q}r^{1-q}}\right) V\left( Q,P\right)
\label{e.2.20} \\
& \leq \frac{q}{2\left( 1-q\right) }\left( \frac{R^{1-q}-r^{1-q}}{%
R^{1-q}r^{1-q}}\right) \chi \left( Q,P\right)  \notag \\
& \leq \frac{q}{4\left( 1-q\right) }\left( \frac{R^{1-q}-r^{1-q}}{%
R^{1-q}r^{1-q}}\right) \left( R-r\right)  \notag
\end{align}%
provided that $Q,P\in S_{1}\left( H\right) ,$ with $P,$ $Q$ invertible and
satisfying the condition (\ref{e.2.16}).
\end{remark}

\section{Other Reverse Inequalities}

Utilising different techniques we can obtain other upper bounds for the
quantum $f$-divergence as follows. Applications for Umegaki relative entropy
and $\chi ^{2}$-divergence are also provided. 

\begin{theorem}
\label{t.2.3}Let $f:[0,\infty )\rightarrow \mathbb{R}$ be a continuous
convex function that is normalized. If $Q,P\in S_{1}\left( H\right) ,$ with $%
P$ invertible, and there exists $R\geq 1\geq r\geq 0$ such that the
condition (\ref{e.2.16}) is satisfied, then 
\begin{equation}
0\leq S_{f}\left( Q,P\right) \leq \frac{\left( R-1\right) f\left( r\right)
+\left( 1-r\right) f\left( R\right) }{R-r}.  \label{e.2.21}
\end{equation}
\end{theorem}

\begin{proof}
By the convexity of $f$ we have%
\begin{equation*}
f\left( t\right) =f\left( \frac{\left( R-t\right) r+\left( t-r\right) R}{R-r}%
\right) \leq \frac{\left( R-t\right) f\left( r\right) +\left( t-r\right)
f\left( R\right) }{R-r}
\end{equation*}%
for any $t\in \left[ r,R\right] .$

This inequality implies the following inequality in the operator order of $%
\mathcal{B}\left( \mathcal{B}_{2}\left( H\right) \right) $%
\begin{equation*}
f\left( \mathfrak{A}_{Q,P}\right) \leq \frac{\left( R1_{\mathcal{B}%
_{2}\left( H\right) }-\mathfrak{A}_{Q,P}\right) f\left( r\right) +\left( 
\mathfrak{A}_{Q,P}-r1_{\mathcal{B}_{2}\left( H\right) }\right) f\left(
R\right) }{R-r},
\end{equation*}%
which can be written as%
\begin{align}
& \left\langle f\left( \mathfrak{A}_{Q,P}\right) T,T\right\rangle _{2}
\label{e.2.22} \\
& \leq \frac{f\left( r\right) }{R-r}\left\langle \left( R1_{\mathcal{B}%
_{2}\left( H\right) }-\mathfrak{A}_{Q,P}\right) T,T\right\rangle _{2}+\frac{%
f\left( R\right) }{R-r}\left\langle \left( \mathfrak{A}_{Q,P}-r1_{\mathcal{B}%
_{2}\left( H\right) }\right) T,T\right\rangle _{2}  \notag
\end{align}%
for any $T\in \mathcal{B}_{2}\left( H\right) .$

This inequality is of interest in itself.

Now, if we take in (\ref{e.2.22}) $T=P^{1/2},$ $P\in S_{1}\left( H\right) ,$
then we get the desired result (\ref{e.2.22}).
\end{proof}

\begin{remark}
\label{r.2.3}If we take in (\ref{e.2.21}) $f\left( t\right) =t^{2}-1,$ then
we get 
\begin{equation}
0\leq \chi ^{2}\left( Q,P\right) \leq \left( R-1\right) \left( 1-r\right) 
\frac{R+r+2}{R-r}  \label{e.2.23}
\end{equation}%
for $Q,P\in S_{1}\left( H\right) ,$ with $P$ invertible and satisfying the
condition (\ref{e.2.16}).

If we take in (\ref{e.2.21}) $f\left( t\right) =t\ln t$, then we get the
inequality%
\begin{equation}
0\leq U\left( Q,P\right) \leq \ln \left[ r^{\frac{\left( R-1\right) r}{R-r}%
}R^{\frac{R\left( 1-r\right) }{R-r}}\right]  \label{e.2.24}
\end{equation}%
provided that $Q,P\in S_{1}\left( H\right) ,$ with $P,$ $Q$ invertible and
satisfying the condition (\ref{e.2.16}).

If we take in (\ref{e.2.21}) $f\left( t\right) =-\ln t$, then we get the
inequality%
\begin{equation}
0\leq U\left( P,Q\right) \leq \ln \left[ r^{\frac{1-R}{R-r}}R^{\frac{r-1}{R-r%
}}\right]  \label{e.2.25}
\end{equation}%
for $Q,P\in S_{1}\left( H\right) ,$ with $P,$ $Q$ invertible and satisfying
the condition (\ref{e.2.16}).
\end{remark}

We also have:

\begin{theorem}
\label{t.2.4}Let $f:[0,\infty )\rightarrow \mathbb{R}$ be a continuous
convex function that is normalized. If $Q,P\in S_{1}\left( H\right) ,$ with $%
P$ invertible, and there exists $R>1>r\geq 0$ such that the condition (\ref%
{e.2.16}) is satisfied, then%
\begin{align}
0& \leq S_{f}\left( Q,P\right) \leq \frac{\left( R-1\right) \left(
1-r\right) }{R-r}\Psi _{f}\left( 1;r,R\right)  \label{e.2.26} \\
& \leq \frac{\left( R-1\right) \left( 1-r\right) }{R-r}\sup_{t\in \left(
r,R\right) }\Psi _{f}\left( t;r,R\right)  \notag \\
& \leq \left( R-1\right) \left( 1-r\right) \frac{f_{-}^{\prime }\left(
R\right) -f_{+}^{\prime }\left( r\right) }{R-r}  \notag \\
& \leq \frac{1}{4}\left( R-r\right) \left[ f_{-}^{\prime }\left( R\right)
-f_{+}^{\prime }\left( r\right) \right]  \notag
\end{align}%
where $\Psi _{f}\left( \cdot ;r,R\right) :\left( r,R\right) \rightarrow 
\mathbb{R}$ is defined by%
\begin{equation}
\Psi _{f}\left( t;r,R\right) =\frac{f\left( R\right) -f\left( t\right) }{R-t}%
-\frac{f\left( t\right) -f\left( r\right) }{t-r}.  \label{e.2.27}
\end{equation}%
We also have 
\begin{align}
0& \leq S_{f}\left( Q,P\right) \leq \frac{\left( R-1\right) \left(
1-r\right) }{R-r}\Psi _{f}\left( 1;r,R\right)  \label{e.2.7.a} \\
& \leq \frac{1}{4}\left( R-r\right) \Psi _{f}\left( 1;r,R\right)  \notag \\
& \leq \frac{1}{4}\left( R-r\right) \sup_{t\in \left( r,R\right) }\Psi
_{f}\left( t;r,R\right)  \notag \\
& \leq \frac{1}{4}\left( R-r\right) \left[ f_{-}^{\prime }\left( R\right)
-f_{+}^{\prime }\left( r\right) \right] .  \notag
\end{align}
\end{theorem}

\begin{proof}
By denoting 
\begin{equation*}
\Delta _{f}\left( t;r,R\right) :=\frac{\left( t-r\right) f\left( R\right)
+\left( R-t\right) f\left( r\right) }{R-r}-f\left( t\right) ,\quad t\in %
\left[ r,R\right]
\end{equation*}%
we have%
\begin{align}
\Delta _{f}\left( t;r,R\right) & =\frac{\left( t-r\right) f\left( R\right)
+\left( R-t\right) f\left( r\right) -\left( R-r\right) f\left( t\right) }{R-r%
}  \label{e.2.28} \\
& =\frac{\left( t-r\right) f\left( R\right) +\left( R-t\right) f\left(
r\right) -\left( T-t+t-r\right) f\left( t\right) }{R-r}  \notag \\
& =\frac{\left( t-r\right) \left[ f\left( R\right) -f\left( t\right) \right]
-\left( R-t\right) \left[ f\left( t\right) -f\left( r\right) \right] }{M-m} 
\notag \\
& =\frac{\left( R-t\right) \left( t-r\right) }{R-r}\Psi _{f}\left(
t;r,R\right)  \notag
\end{align}%
for any $t\in \left( r,R\right) .$

From the proof of Theorem \ref{t.2.3} we have%
\begin{align}
& \left\langle f\left( \mathfrak{A}_{Q,P}\right) T,T\right\rangle _{2}
\label{e.2.29} \\
& \leq \frac{f\left( r\right) }{R-r}\left\langle \left( R1_{\mathcal{B}%
_{2}\left( H\right) }-\mathfrak{A}_{Q,P}\right) T,T\right\rangle _{2}+\frac{%
f\left( R\right) }{R-r}\left\langle \left( \mathfrak{A}_{Q,P}-r1_{\mathcal{B}%
_{2}\left( H\right) }\right) T,T\right\rangle _{2}  \notag \\
& =\frac{\left( \left\langle \mathfrak{A}_{Q,P}T,T\right\rangle
_{2}-r\right) f\left( R\right) +\left( R-\left\langle \mathfrak{A}%
_{Q,P}T,T\right\rangle _{2}\right) f\left( r\right) }{R-r}  \notag
\end{align}%
for any $T\in \mathcal{B}_{2}\left( H\right) ,$ $\left\Vert T\right\Vert
_{2}=1.$

This implies that%
\begin{align}
0& \leq \left\langle f\left( \mathfrak{A}_{Q,P}\right) T,T\right\rangle
_{2}-f\left( \left\langle \mathfrak{A}_{Q,P}T,T\right\rangle _{2}\right)
\label{e.2.30} \\
& \leq \frac{\left( \left\langle \mathfrak{A}_{Q,P}T,T\right\rangle
_{2}-r\right) f\left( R\right) +\left( R-\left\langle \mathfrak{A}%
_{Q,P}T,T\right\rangle _{2}\right) f\left( r\right) }{R-r}-f\left(
\left\langle \mathfrak{A}_{Q,P}T,T\right\rangle _{2}\right)  \notag \\
& =\Delta _{f}\left( \left\langle \mathfrak{A}_{Q,P}T,T\right\rangle
_{2};r,R\right)  \notag \\
& =\frac{\left( R-\left\langle \mathfrak{A}_{Q,P}T,T\right\rangle
_{2}\right) \left( \left\langle \mathfrak{A}_{Q,P}T,T\right\rangle
_{2}-r\right) }{R-r}\Psi _{f}\left( \left\langle \mathfrak{A}%
_{Q,P}T,T\right\rangle _{2};r,R\right)  \notag
\end{align}%
for any $T\in \mathcal{B}_{2}\left( H\right) ,$ $\left\Vert T\right\Vert
_{2}=1.$

Since 
\begin{align}
\Psi _{f}\left( \left\langle \mathfrak{A}_{Q,P}T,T\right\rangle
_{2};r,R\right) & \leq \sup_{t\in \left( r,R\right) }\Psi _{f}\left(
t;r,R\right)  \label{e.2.31} \\
& =\sup_{t\in \left( r,R\right) }\left[ \frac{f\left( R\right) -f\left(
t\right) }{R-t}-\frac{f\left( t\right) -f\left( r\right) }{t-r}\right] 
\notag \\
& \leq \sup_{t\in \left( r,R\right) }\left[ \frac{f\left( R\right) -f\left(
t\right) }{R-t}\right] +\sup_{t\in \left( r,R\right) }\left[ -\frac{f\left(
t\right) -f\left( r\right) }{t-r}\right]  \notag \\
& =\sup_{t\in \left( r,R\right) }\left[ \frac{f\left( R\right) -f\left(
t\right) }{R-t}\right] -\inf_{t\in \left( r,R\right) }\left[ \frac{f\left(
t\right) -f\left( r\right) }{t-r}\right]  \notag \\
& =f_{-}^{\prime }\left( R\right) -f_{+}^{\prime }\left( r\right) ,  \notag
\end{align}%
and, obviously%
\begin{equation}
\frac{1}{R-r}\left( R-\left\langle \mathfrak{A}_{Q,P}T,T\right\rangle
_{2}\right) \left( \left\langle \mathfrak{A}_{Q,P}T,T\right\rangle
_{2}-r\right) \leq \frac{1}{4}\left( R-r\right) ,  \label{e.2.32}
\end{equation}%
then by (\ref{e.2.30})-(\ref{e.2.32}) we have%
\begin{align}
0& \leq \left\langle f\left( \mathfrak{A}_{Q,P}\right) T,T\right\rangle
_{2}-f\left( \left\langle \mathfrak{A}_{Q,P}T,T\right\rangle _{2}\right)
\label{e.2.33} \\
& \leq \frac{\left( R-\left\langle \mathfrak{A}_{Q,P}T,T\right\rangle
_{2}\right) \left( \left\langle \mathfrak{A}_{Q,P}T,T\right\rangle
_{2}-r\right) }{R-r}\Psi _{f}\left( \left\langle \mathfrak{A}%
_{Q,P}T,T\right\rangle _{2};r,R\right)  \notag \\
& \leq \frac{\left( R-\left\langle \mathfrak{A}_{Q,P}T,T\right\rangle
_{2}\right) \left( \left\langle \mathfrak{A}_{Q,P}T,T\right\rangle
_{2}-r\right) }{R-r}\sup_{t\in \left( r,R\right) }\Psi _{f}\left(
t;r,R\right)  \notag \\
& \leq \left( R-\left\langle \mathfrak{A}_{Q,P}T,T\right\rangle _{2}\right)
\left( \left\langle \mathfrak{A}_{Q,P}T,T\right\rangle _{2}-r\right) \frac{%
f_{-}^{\prime }\left( R\right) -f_{+}^{\prime }\left( r\right) }{R-r}  \notag
\\
& \leq \frac{1}{4}\left( R-r\right) \left[ f_{-}^{\prime }\left( R\right)
-f_{+}^{\prime }\left( r\right) \right]  \notag
\end{align}%
for any $T\in \mathcal{B}_{2}\left( H\right) ,$ $\left\Vert T\right\Vert
_{2}=1.$

This inequality is of interest in itself.

Now, if we take in (\ref{e.2.33}) $T=P^{1/2},$ then we get the desired
result (\ref{e.2.26}).

The inequality (\ref{e.2.7.a}) is obvious from \ (\ref{e.2.26}).
\end{proof}

\begin{remark}
\label{r.2.4}If we consider the convex normalized function $f\left( t\right)
=t^{2}-1,$ then%
\begin{equation*}
\Psi _{f}\left( t;r,R\right) =\frac{R^{2}-t^{2}}{R-t}-\frac{t^{2}-r^{2}}{t-r}%
=R-r,\text{ }t\in \left( r,R\right)
\end{equation*}%
and we get from (\ref{e.2.26}) the simple inequality%
\begin{equation}
0\leq \chi ^{2}\left( Q,P\right) \leq \left( R-1\right) \left( 1-r\right)
\label{e.2.34}
\end{equation}%
for $Q,P\in S_{1}\left( H\right) ,$ with $P$ invertible and satisfying the
condition (\ref{e.2.16}), which is better than (\ref{e.2.23}).

If we take the convex normalized function $f\left( t\right) =t^{-1}-1,$ then
we have%
\begin{equation*}
\Psi _{f}\left( t;r,R\right) =\frac{R^{-1}-t^{-1}}{R-t}-\frac{t^{-1}-r^{-1}}{%
t-r}=\frac{R-r}{rRt},\text{ }t\in \left[ r,R\right] .
\end{equation*}%
Also 
\begin{equation*}
S_{f}\left( Q,P\right) =\chi ^{2}\left( P,Q\right) .
\end{equation*}%
Using (\ref{e.2.26}) we get%
\begin{equation}
0\leq \chi ^{2}\left( P,Q\right) \leq \frac{\left( R-1\right) \left(
1-r\right) }{Rr}  \label{e.2.35}
\end{equation}%
for $Q,P\in S_{1}\left( H\right) ,$ with $Q$ invertible and satisfying the
condition (\ref{e.2.16}).

If we consider the convex function $f\left( t\right) =-\ln t$ defined on $%
\left[ r,R\right] \subset \left( 0,\infty \right) ,$ then 
\begin{eqnarray*}
\Psi _{f}\left( t;r,R\right) &=&\frac{-\ln R+\ln t}{R-t}-\frac{-\ln t+\ln r}{%
t-r} \\
&=&\frac{\left( R-r\right) \ln t-\left( R-t\right) \ln r-\left( t-r\right)
\ln R}{\left( M-t\right) \left( t-m\right) } \\
&=&\ln \left( \frac{t^{R-r}}{r^{R-t}M^{t-r}}\right) ^{\frac{1}{\left(
R-t\right) \left( t-r\right) }},\text{ }t\in \left( r,R\right) .
\end{eqnarray*}%
Then by (\ref{e.2.26}) we have%
\begin{equation}
0\leq U\left( P,Q\right) \leq \ln \left[ r^{\frac{1-R}{R-r}}R^{\frac{r-1}{R-r%
}}\right] \leq \frac{\left( R-1\right) \left( 1-r\right) }{rR}
\label{e.2.36}
\end{equation}%
for $Q,P\in S_{1}\left( H\right) ,$ with $P,$ $Q$ invertible and satisfying
the condition (\ref{e.2.16}).

If we consider the convex function $f\left( t\right) =t\ln t$ defined on $%
\left[ r,R\right] \subset \left( 0,\infty \right) ,$ then%
\begin{equation*}
\Psi _{f}\left( t;r,R\right) =\frac{R\ln R-t\ln t}{R-t}-\frac{t\ln t-r\ln r}{%
t-r},\text{ }t\in \left( r,R\right) ,
\end{equation*}%
which gives that%
\begin{equation*}
\Psi _{f}\left( 1;r,R\right) =\frac{R\ln R}{R-1}-\frac{r\ln r}{1-r}.
\end{equation*}%
Using (\ref{e.2.26}) we get 
\begin{align}
0& \leq U\left( Q,P\right) \leq \ln \left[ R^{\frac{\left( 1-r\right) R}{R-r}%
}r^{\frac{\left( 1-R\right) r}{R-r}}\right]  \label{e.2.37} \\
& \leq \left( R-1\right) \left( 1-r\right) \ln \left[ \left( \frac{R}{r}%
\right) ^{\frac{1}{R-r}}\right]  \notag
\end{align}%
for $Q,P\in S_{1}\left( H\right) ,$ with $P,$ $Q$ invertible and satisfying
the condition (\ref{e.2.16}).
\end{remark}

We also have:

\begin{theorem}
\label{t.2.5}Let $f:[0,\infty )\rightarrow \mathbb{R}$ be a continuous
convex function that is normalized. If $Q,P\in S_{1}\left( H\right) ,$ with $%
P$ invertible, and there exists $R>1>r\geq 0$ such that the condition (\ref%
{e.2.16}) is satisfied, then%
\begin{equation}
0\leq S_{f}\left( Q,P\right) \leq 2\left[ \frac{f\left( r\right) +f\left(
R\right) }{2}-f\left( \frac{r+R}{2}\right) \right] .  \label{e.2.38}
\end{equation}
\end{theorem}

\begin{proof}
We recall the following result (see for instance \cite{SSDBul})\ that
provides a refinement and a reverse for the weighted Jensen's discrete
inequality:%
\begin{align}
& n\min_{i\in \left\{ 1,...,n\right\} }\left\{ p_{i}\right\} \left[ \frac{1}{%
n}\sum_{i=1}^{n}f\left( x_{i}\right) -f\left( \frac{1}{n}\sum_{i=1}^{n}x_{i}%
\right) \right]  \label{e.2.39} \\
& \leq \frac{1}{P_{n}}\sum_{i=1}^{n}p_{i}f\left( x_{i}\right) -f\left( \frac{%
1}{P_{n}}\sum_{i=1}^{n}p_{i}x_{i}\right)  \notag \\
& \leq n\max_{i\in \left\{ 1,...,n\right\} }\left\{ p_{i}\right\} \left[ 
\frac{1}{n}\sum_{i=1}^{n}f\left( x_{i}\right) -f\left( \frac{1}{n}%
\sum_{i=1}^{n}x_{i}\right) \right] ,  \notag
\end{align}%
where $f:C\rightarrow \mathbb{R}$ is a convex function defined on the convex
subset $C$ of the linear space $X,$ $\left\{ x_{i}\right\} _{i\in \left\{
1,...,n\right\} }\subset C$ are vectors and $\left\{ p_{i}\right\} _{i\in
\left\{ 1,...,n\right\} }$ are nonnegative numbers with $P_{n}:=%
\sum_{i=1}^{n}p_{i}>0.$

For $n=2$ we deduce from (\ref{e.2.8}) that%
\begin{align}
& 2\min \left\{ s,1-s\right\} \left[ \frac{f\left( x\right) +f\left(
y\right) }{2}-f\left( \frac{x+y}{2}\right) \right]  \label{e.3.40} \\
& \leq sf\left( x\right) +\left( 1-s\right) f\left( y\right) -f\left(
sx+\left( 1-s\right) y\right)  \notag \\
& \leq 2\max \left\{ s,1-s\right\} \left[ \frac{f\left( x\right) +f\left(
y\right) }{2}-f\left( \frac{x+y}{2}\right) \right]  \notag
\end{align}%
for any $x,y\in C$ and $s\in \left[ 0,1\right] .$

Now, if we use the second inequality in (\ref{e.3.40}) for $x=r,$ $y=R,$ $s=%
\frac{R-t}{R-r}$ with $t\in \left[ r,R\right] ,$ then we have%
\begin{align}
& \frac{\left( R-t\right) f\left( r\right) +\left( t-r\right) f\left(
R\right) }{R-r}-f\left( t\right)  \label{e.3.41} \\
& \leq 2\max \left\{ \frac{R-t}{R-r},\frac{t-r}{R-r}\right\} \left[ \frac{%
f\left( r\right) +f\left( R\right) }{2}-f\left( \frac{r+R}{2}\right) \right]
\notag \\
& =\left[ 1+\frac{2}{R-r}\left\vert t-\frac{r+R}{2}\right\vert \right] \left[
\frac{f\left( r\right) +f\left( R\right) }{2}-f\left( \frac{r+R}{2}\right) %
\right]  \notag
\end{align}%
for any $t\in \left[ r,R\right] .$

This implies in the operator order of $\mathcal{B}\left( \mathcal{B}%
_{2}\left( H\right) \right) $%
\begin{align*}
& \frac{\left( R1_{\mathcal{B}_{2}\left( H\right) }-\mathfrak{A}%
_{Q,P}\right) f\left( r\right) +\left( \mathfrak{A}_{Q,P}-r1_{\mathcal{B}%
_{2}\left( H\right) }\right) f\left( R\right) }{R-r}-f\left( \mathfrak{A}%
_{Q,P}\right) \\
& \leq \left[ \frac{f\left( r\right) +f\left( R\right) }{2}-f\left( \frac{r+R%
}{2}\right) \right] \\
& \times \left[ 1_{\mathcal{B}_{2}\left( H\right) }+\frac{2}{R-r}\left\vert 
\mathfrak{A}_{Q,P}-\frac{r+R}{2}1_{\mathcal{B}_{2}\left( H\right)
}\right\vert \right]
\end{align*}%
which implies that%
\begin{align}
0& \leq \left\langle f\left( \mathfrak{A}_{Q,P}\right) T,T\right\rangle
_{2}-f\left( \left\langle \mathfrak{A}_{Q,P}T,T\right\rangle _{2}\right)
\label{e.3.42} \\
& \leq \frac{\left( \left\langle \mathfrak{A}_{Q,P}T,T\right\rangle
_{2}-r\right) f\left( R\right) +\left( R-\left\langle \mathfrak{A}%
_{Q,P}T,T\right\rangle _{2}\right) f\left( r\right) }{R-r}-f\left(
\left\langle \mathfrak{A}_{Q,P}T,T\right\rangle _{2}\right)  \notag \\
& \leq \left[ \frac{f\left( r\right) +f\left( R\right) }{2}-f\left( \frac{r+R%
}{2}\right) \right]  \notag \\
& \times \left[ 1+\frac{2}{R-r}\left\langle \left\vert \mathfrak{A}_{Q,P}-%
\frac{r+R}{2}1_{\mathcal{B}_{2}\left( H\right) }\right\vert T,T\right\rangle
_{2}\right]  \notag \\
& \leq 2\left[ \frac{f\left( r\right) +f\left( R\right) }{2}-f\left( \frac{%
r+R}{2}\right) \right]  \notag
\end{align}%
for any $T\in \mathcal{B}_{2}\left( H\right) ,$ $\left\Vert T\right\Vert
_{2}=1.$

This is an inequality of interest in itself.

If we take in (\ref{e.3.42}) $T=P^{1/2},$ $P\in S_{1}\left( H\right) ,$ then
we get the desired result (\ref{e.2.38}).
\end{proof}

\begin{remark}
\label{r.2.5}If we take $f\left( t\right) =t^{2}-1$ in (\ref{e.2.38}), then
we get%
\begin{equation*}
0\leq \chi ^{2}\left( Q,P\right) \leq \frac{1}{2}\left( R-r\right) ^{2}
\end{equation*}%
for $Q,P\in S_{1}\left( H\right) ,$ with $P$ invertible and satisfying the
condition (\ref{e.2.16}), which is not as good as (\ref{e.2.34}).

If we take in (\ref{e.2.38}) $f\left( t\right) =t^{-1}-1,$ then we have%
\begin{equation}
0\leq \chi ^{2}\left( P,Q\right) \leq \frac{\left( R-r\right) ^{2}}{rR\left(
r+R\right) }  \label{e.2.43}
\end{equation}%
for $Q,P\in S_{1}\left( H\right) ,$ with $P$ invertible and satisfying the
condition (\ref{e.2.16}).

If we take in (\ref{e.2.38}) $f\left( t\right) =-\ln t,$ then we have%
\begin{equation}
0\leq U\left( P,Q\right) \leq \ln \left( \frac{\left( R+r\right) ^{2}}{4rR}%
\right)  \label{e.2.44}
\end{equation}%
for $Q,P\in S_{1}\left( H\right) ,$ with $P$ invertible and satisfying the
condition (\ref{e.2.16}).

From (\ref{e.2.19}) we have the following absolute upper bound%
\begin{equation}
0\leq U\left( P,Q\right) \leq \frac{\left( R-r\right) ^{2}}{4rR}
\label{e.2.45}
\end{equation}%
for $Q,P\in S_{1}\left( H\right) ,$ with $P$ invertible and satisfying the
condition (\ref{e.2.16}).

Utilising the elementary inequality $\ln x\leq x-1,$ $x>0,$ we have that%
\begin{equation*}
\ln \left( \frac{\left( R+r\right) ^{2}}{4rR}\right) \leq \frac{\left(
R-r\right) ^{2}}{4rR},
\end{equation*}%
which shows that (\ref{e.2.44}) is better than (\ref{e.2.45}).
\end{remark}

\end{document}